\documentclass[a4paper,leqno]{amsart}
\usepackage{graphicx, amsmath, amssymb, amsthm, hyperref, verbatim, multicol, mathrsfs, tikz, caption, subcaption}

\theoremstyle{definition}
\newtheorem{thm}{Theorem}[section]

\newtheorem{defi}[thm]{Definition}
\newtheorem{lemm}[thm]{Lemma}
\newtheorem{cor}[thm]{Corollary}
\newtheorem{rem}[thm]{Remark}

\newtheorem{prop}[thm]{Proposition}

\theoremstyle{remark} 
\newtheorem{case}{Case}
\newtheorem{step}{Step}

\DeclareMathOperator{\Mod}{mod}
\DeclareMathOperator*{\osc}{osc}
\DeclareMathOperator{\diam}{diam}

\begin{document}
	\title[Reciprocal lower bound on modulus of curve families]{Reciprocal lower bound on modulus of curve families in metric surfaces} 
	
	\author{Kai Rajala} 
	\address{Department of Mathematics and Statistics, University of Jyv\"askyl\"a, P.O. Box 35 (MaD), FI-40014, University of Jyv\"askyl\"a, Finland.}
	\email{kai.i.rajala@jyu.fi, matthew.d.romney@jyu.fi}
	
	\author{Matthew Romney}
	
	\thanks{The first author was supported by the Academy of Finland, project number 308659. The second author was partially supported by the Academy of Finland grant 288501 and by the ERC Starting Grant 713998 GeoMeG. \newline {\it 2010 Mathematics Subject Classification.} Primary 30L10, Secondary 30C65, 28A75.} 
	
	\begin{abstract}
		We prove that any metric space $X$ homeomorphic to $\mathbb{R}^2$ with locally finite Hausdorff 2-measure satisfies a reciprocal lower bound on modulus of curve families associated to a quadrilateral. More precisely, let $Q \subset X$ be a topological quadrilateral with boundary edges (in cyclic order) denoted by $\zeta_1, \zeta_2, \zeta_3, \zeta_4$ and let $\Gamma(\zeta_i, \zeta_j; Q)$ denote the family of curves in $Q$ connecting $\zeta_i$ and $\zeta_j$; then $\Mod \Gamma(\zeta_1, \zeta_3; Q) \Mod \Gamma(\zeta_2, \zeta_4; Q) \geq 1/\kappa$ for $\kappa = 2000^2\cdot (4/\pi)^2$.  This answers a question in \cite{Raj:16} concerning minimal hypotheses under which a metric space admits a quasiconformal parametrization by a domain in $\mathbb{R}^2$. 
	\end{abstract}
	
      	\maketitle 
	
	\section{Introduction} 
	
	The classical uniformization theorem states that any simply connected Riemann surface can be mapped onto either the Euclidean plane $\mathbb{R}^2$, the sphere $\mathbb{S}^2$, or the unit disk $\mathbb{D}$ by a conformal mapping. For obtaining similar results in the setting of metric spaces, the class of conformal mappings is too restrictive and it is natural to consider instead some type of quasiconformal mapping. One such class is {\it quasisymmetric mappings}, and a large body of recent literature is dedicated to quasisymmetric uniformization of metric spaces. We mention specifically papers by Semmes \cite{Sem:96b} and Bonk--Kleiner \cite{BonkKle:02} as important references.
	
	Another approach is to use the so-called {\it geometric definition} of quasiconformal mappings, based on the notion of modulus of a curve family. In the recent paper \cite{Raj:16}, the first-named author proves a version of the uniformization theorem for  metric spaces homeomorphic to $\mathbb{R}^2$ with locally finite Hausdorff 2-measure. In the present paper, we call such spaces {\it metric surfaces}. 
	
	In \cite{Raj:16} a condition on metric surfaces called {\it reciprocality} (see Definition \ref{defi:reciprocality} below) is introduced and shown to be necessary and sufficient for the existence of a quasiconformal parametrization by a domain in $\mathbb{R}^2$. We refer the reader to the introduction of \cite{Raj:16} for a detailed overview of the problem and additional references to the literature. 
	
	In this paper, we show that one part of the definition of reciprocality is satisfied by all metric surfaces and therefore is unnecessary. This result gives a positive answer to Question 17.5 from \cite{Raj:16}. 
	
	We first recall the relevant definitions and establish some notation. Let $(X,d,\mu)$ be a metric measure space. For a family $\Gamma$ of curves in $X$, the {\it $p$-modulus} of $\Gamma$ is defined as 
	$$\Mod_p \Gamma = \inf \int_X \rho^p\,d\mu ,$$
	where the infimum is taken over all Borel functions $\rho: X \rightarrow [0,\infty]$ with the property that $\int_{\gamma} \rho\,ds \geq 1$ for all locally rectifiable curves $\gamma \in \Gamma$. Such a function $\rho$ is called {\it admissible}. If the exponent $p$ is understood, a homeomorphism $f: (X,d,\mu) \rightarrow (Y,d',\nu)$ between metric measure spaces is {\it quasiconformal} if there exists $K \geq 1$ such that
	$$K^{-1} \Mod_p \Gamma \leq \Mod_p f(\Gamma) \leq K \Mod_p \Gamma $$
	for all curve families $\Gamma$ in $X$. In this paper, we always take $p=2$ and assume that a metric space $(X,d)$ is equipped with the Hausdorff 2-measure $\mathcal{H}^2$, and we write $\Mod \Gamma$ in place of $\Mod_2 \Gamma$.

	Throughout this paper, we assume that $(X,d)$ is a metric surface as defined above. A {\it quadrilateral} in $X$ is a subset $Q \subset X$ homeomorphic to  $[0,1]^2$ with four designated non-overlapping boundary arcs, denoted in cyclic order by $\zeta_1$, $\zeta_2$, $\zeta_3$, $\zeta_4$, which are the images of $[0,1] \times\{0\}$, $\{1\} \times [0,1]$, $[0,1] \times \{1\}$ and $\{0\} \times [0,1]$, respectively, under the parametrizing homeomorphism from $[0,1]^2$. We write $\Gamma_1(Q)$ to denote the family $\Gamma(\zeta_1,\zeta_3; Q)$ of curves in $Q$ connecting $\zeta_1$ and $\zeta_3$, and $\Gamma_2(Q)$ to denote the family $\Gamma(\zeta_2,\zeta_4; Q)$ of curves in $Q$ connecting $\zeta_2$ and $\zeta_4$. More generally, for disjoint closed sets $E,F$ contained in the set $G \subset X$, the notation $\Gamma(E,F;G)$ is used to denote the family of curves in $G$ which intersect both $E$ and $F$.

	\begin{defi} \label{defi:reciprocality}
		The metric surface $(X, d)$  is {\it reciprocal} if there exists $\kappa \geq 1$ such that for all quadrilaterals $Q$ in $X$,
		\begin{equation} \label{equ:reciprocality(1)}
		\Mod \Gamma_1(Q) \Mod \Gamma_2(Q) \leq \kappa
		\end{equation}
		and
		\begin{equation} \label{equ:reciprocality(2)}
		\Mod \Gamma_1(Q) \Mod \Gamma_2(Q) \geq 1/\kappa,
		\end{equation} 
		and for all $x \in X$ and $R>0$ such that $X \setminus B(x,R) \neq \emptyset$,
		\begin{equation} \label{equ:reciprocality(3)} 
		\lim_{r \rightarrow 0} \Mod \Gamma(B(x,r), X \setminus B(x,R); B(x,R)) = 0.
		\end{equation} 
	\end{defi} 
	
	We then have the following result. 
	\begin{thm}[\cite{Raj:16}, Thm. 1.4] \label{thm:uniformization}
		There exists a domain $\Omega \subset \mathbb{R}^2$ and a quasiconformal mapping $f: (X,d) \rightarrow \Omega$ if and only if $X$ is reciprocal. 
	\end{thm}
	
	The necessity of each condition in Definition \ref{defi:reciprocality} is immediate; standard computations show that $\mathbb{R}^2$ is reciprocal. The actual content of Theorem \ref{thm:uniformization} is that these conditions are sufficient to construct ``by hand'' a mapping that can then be shown to be quasiconformal. However, the question of whether a weaker set of assumptions might still be sufficient to construct such a quasiconformal mapping is not fully settled in \cite{Raj:16}. 
	
	It is not difficult to construct examples of metric surfaces for which conditions \eqref{equ:reciprocality(1)} and \eqref{equ:reciprocality(3)} fail. For instance, the quotient space $\mathbb{R}^2/ \sim$, where $x \sim y$ if $x=y$ or if 
	both $x$ and $y$ belong to the closed unit disc, has a natural metric for which both conditions fail. On the other hand, it was conjectured in \cite{Raj:16} (Question 17.5) that in fact condition \eqref{equ:reciprocality(2)} holds for all $(X,d)$. 
	The main result of this paper shows that this is indeed the case. 
	
	\begin{thm} \label{thm:main}
		Let $(X,d)$ be a metric space homeomorphic to $\mathbb{R}^2$ with locally finite Hausdorff 2-measure. There exists a constant $\kappa\geq 1$, independent of $X$, such that $\Mod \Gamma_1(Q) \Mod \Gamma_2(Q) \geq 1/\kappa$ for all quadrilaterals $Q \subset X$.
	\end{thm}
	
	As a consequence of Theorem \ref{thm:main}, condition \eqref{equ:reciprocality(2)} in Definition \ref{defi:reciprocality} is unnecessary. Our proof as written gives a value of $\kappa = 2000^2\cdot (4/\pi)^2$, though optimizing each step would improve this to $\kappa = 216^2\cdot (4/\pi)^2$. 
	It is a corollary of Theorem 1.5 in \cite{Raj:16}, as improved in \cite{Rom:17}, that if $X$ is reciprocal (and hence $X$ admits a quasiconformal parametrization), then Theorem \ref{thm:main} holds with $\kappa = (4/\pi)^2$.  For this reason, it is natural to conjecture that the best possible $\kappa$ for the general case is also $(4/\pi)^2$, though our techniques fall far short of this.
	
	In Proposition 15.8 of \cite{Raj:16}, Theorem \ref{thm:main} (with a larger value of $\kappa$) is proved under the assumption that $X$ satisfies the mass upper bound $\mathcal{H}^2(B(x,r)) \leq Cr^2$ for some $C>0$ independent of $x$ and $r$. Our proof follows a similar outline; the difficulty is to avoid using the upper bound. 
	
	The basic approach is to construct an ``energy-minimizing'' or ``harmonic'' function $u: Q \rightarrow [0,\infty)$ which satisfies the boundary constraints $u|\zeta_1 = 0$ and $u|\zeta_3 = 1$. Working only from the assumptions at hand, one can establish relevant properties of $u$. The main property needed to prove Theorem \ref{thm:main} is that a version of the coarea inequality holds for $u$. For the case when $X$ satisfies the mass upper bound $\mathcal{H}^2(B(x,r)) \leq Cr^2$, this is found in Proposition 15.7 of \cite{Raj:16}. The coarea inequality implies that, from the level sets of $u$, one may extract a large family of rectifiable curves contained in $\Gamma_2(Q)$. Since $u$ is defined by means of the curve family $\Gamma_1(Q)$, this provides the necessary link between $\Gamma_1(Q)$ and $\Gamma_2(Q)$. Roughly speaking, if there are few curves in $\Gamma_1(Q)$, as quantified by modulus, then these corresponding curves in $\Gamma_2(Q)$ must be short, which implies that the modulus of $\Gamma_2(Q)$ is large. 
	
	The organization of the paper is the following. Section \ref{sec:preliminaries} contains some basic notation and background, including an overview of the construction of the harmonic function $u$ described in the previous paragraph. In Section \ref{sec:level_sets}, we prove several properties of the level sets of $u$ which are required for the proof of Theorem \ref{thm:main}. This section expands on the material present in Section 4 of \cite{Raj:16}. Section \ref{sec:lower_bound} contains the main technical portion of our paper, the coarea inequality for $u$ described previously valid for all metric surfaces, as well as the proof of Theorem \ref{thm:main}. Section \ref{sec:continuity_u} contains a final auxiliary result, namely that the harmonic function $u$ is continuous in general. The continuity of $u$ had previously been proved as Theorem 5.1. of \cite{Raj:16} using the reciprocality conditions \eqref{equ:reciprocality(2)} and \eqref{equ:reciprocality(3)}.
	
	
	\section{Preliminaries} \label{sec:preliminaries}
	
	In this section, we give a review of notation and auxiliary results from \cite{Raj:16} that will be needed. For the remainder of this paper, we let $X$ be a metric surface and $Q$ denote a fixed quadrilateral in $X$. We write $\Gamma_1$ for $\Gamma_1(Q)$. We assume throughout this paper that all curves are non-constant.  
	
	For $k \in \{1,2\}$ and $\varepsilon > 0$, the {\it $k$-dimensional Hausdorff $\varepsilon$-content} of a set $E \subset X$, denoted by $\mathcal{H}_\varepsilon^k(E)$, is defined as
	$$\mathcal{H}_\varepsilon^k(E) = \inf \left\{ \sum a_k \diam(A_j)^k: E \subset \bigcup_{j=1}^\infty A_j, \diam A_j < \varepsilon \right\},  $$
	with normalizing constants $a_1 = 1$ and $a_2 = \pi/4$. The {\it Hausdorff $k$-measure} of $E$ is defined as $\mathcal{H}^k(E) = \lim_{\varepsilon \rightarrow 0} \mathcal{H}_\varepsilon^k(E)$.
	
	We proceed with an overview of the construction of the harmonic function $u$ corresponding to the curve family $\Gamma_1$, as given in Section 4 of \cite{Raj:16}. By a standard argument using Mazur's lemma, there exists a sequence of admissible functions $(\rho_k)$ for $\Gamma_1$ that converges strongly in $L^2$ to a function $\rho \in L^2(Q)$ satisfying $\int_Q \rho^2\,d\mathcal{H}^2 = \Mod \Gamma_1$. By Fuglede's lemma, 
	\begin{equation} \label{equ:fuglede}
	\int_\gamma \rho_k\, ds \rightarrow \int_{\gamma} \rho\, ds < \infty
	\end{equation} 
	for all curves $\gamma$ in $Q$ except for a family of modulus zero. In particular, this implies that $\rho$ is weakly admissible for $\Gamma_1$ (that is, admissible after removing from $\Gamma_1$ a subfamily of modulus zero). We extend the definition of $\rho$ to the entire space $X$ by setting $\rho(x) = 0$ for all $x \in X \setminus Q$.
	
	Let $\Gamma_0$ be the family of curves in $Q$ with a subcurve on which \eqref{equ:fuglede} does not hold. Note that $\Mod \Gamma_0 = 0$.  We define the function $u$ as follows. Let $x \in Q$. If there exists a curve $\gamma \in \Gamma_1 \setminus \Gamma_0$ whose image contains $x$, then define
	\begin{equation} \label{equ:u_definition}
	u(x) = \inf \int_{\gamma_x} \rho\,ds,
	\end{equation}
	where the infimum is taken over all such curves $\gamma$ and over all subcurves $\gamma_x$ of $\gamma$ joining $\zeta_1$ and $x$. Otherwise, define $u(x)$ by
	$$u(x) = \liminf_{y \in E, y \rightarrow x} u(y),$$
	where $E$ is the set of those $y \in Q$ such that $u(y)$ is defined by \eqref{equ:u_definition}. Lemma 4.1 of \cite{Raj:16} shows that $u$ is well-defined in $Q$. 
	
	We recall Lemma 4.3 of \cite{Raj:16}, which states that $\rho$ is a weak upper gradient of $u$. More precisely, 
	\begin{equation} \label{equ:upper_gradient}
	|u(x) - u(y)| \leq \int_{\gamma} \rho\,ds
	\end{equation} 
	for all curves $\gamma$ in $Q$ with $\gamma \notin \Gamma_0$. In particular, $u$ is absolutely continuous along any curve $\gamma \notin \Gamma_0$. We also recall Lemma 4.5 of \cite{Raj:16}, where it is shown that $0 \leq u(x) \leq 1$ for all $x \in Q$. It follows from \eqref{equ:u_definition} that if $x \in \zeta_3$ lies in the image of a curve $\gamma \in \Gamma_1 \setminus \Gamma_0$, then $u(x) \geq 1$ and thus $u(x) = 1$.
	
	As final points of notation, for a set $A \subset Q$, let $\osc_{A} u = \sup_{x,y \in A} |u(x) - u(y)|$. Let $|\gamma|$ denote the image of the curve $\gamma$ in $Q$.
	
	To study the harmonic function $u$, there are three auxiliary results which are employed repeatedly in \cite{Raj:16} and which we state here for easy reference. The first concerns the existence of rectifiable curves and can be found as Proposition 15.1 of \cite{Sem:96c}.
	
	\begin{prop} \label{prop:existence_paths}
		Let $x,y \in X$ be given, $x \neq y$. Suppose that $E \subset X$ is a continuum with $\mathcal{H}^1(E) < \infty$ and $x, y \in E$. Then there is an $L>0$, $L \leq \mathcal{H}^1(E)$, and an injective 1-Lipschitz mapping $\gamma\colon [0,L] \rightarrow X$ such that $\gamma(t) \in E$ for all $t$, $\gamma(0) = x$, $\gamma(L) = y$ and $\mathcal{H}^1(\gamma(F)) = \mathcal{H}^1(F)$ for all measurable sets $F \subset [0,L]$. 
	\end{prop} 
	
	The next is the standard coarea inequality for Lipschitz functions on metric spaces, found in \cite[Proposition 3.1.5]{AmbTil:04}. 
	
	\begin{prop}[Coarea inequality] \label{prop:coarea}
		Let $A \subset X$ be Borel measurable. If $m\colon X \rightarrow \mathbb{R}$ is $L$-Lipschitz and $g\colon A \rightarrow [0, \infty]$ is Borel measurable, then
		$$\int_{\mathbb{R}} \int_{A \cap m^{-1}(t)} g(s)\, d\mathcal{H}^1(s)\, dt \leq \frac{4L}{\pi} \int_A g(x)\, d\mathcal{H}^2(x).$$
	\end{prop}
	
	We also need a topological lemma, cf. \cite[IV Theorem 26]{Moo:62}. 
	\begin{lemm}\label{lemm:separating_continuum}
		Let $A,B \subset Q$ be non-empty sets, and let $K \subset Q$ be a compact set such that $A$ and $B$ belong to different components of $Q \setminus K$. Then there is a continuum 
		$F \subset K$ such that $A$ and $B$ belong to different components of $Q \setminus F$. Moreover, if $\mathcal{H}^1(K) < \infty$ and the component of $Q \setminus K$ containing $A$ is contained in the interior of $Q$, then $F$ may be taken to be the image of an injective Lipschitz mapping $\gamma: \mathbb{S}^1 \rightarrow K$.
	\end{lemm}
	
	\section{Level sets of $u$} \label{sec:level_sets}
	
	In this section, we prove a number of topological properties for the level sets of the harmonic function $u$, or, more precisely, for the closure of these level sets. This section can be viewed as an extension of Section 4 in \cite{Raj:16}, which also studies those properties of $u$ which can be proved without any use of the reciprocality conditions.
	
	The primary technical difficulty we must deal with is that, without assuming the reciprocality conditions, we do not know {\it a priori} that the function $u$ is continuous. However, it is shown in Lemma 4.6 of \cite{Raj:16} that $u$ satisfies a maximum and a minimum principle. To state it, we use the following notation. For an open set $\Omega \subset X$, or a relatively open set $\Omega \subset Q$, let
	$$\partial_* \Omega = (\partial \Omega \cap Q) \cup (\overline{\Omega} \cap (\zeta_1 \cup \zeta_3)) .$$
	Then we have the following.
	
	\begin{lemm}[ Maximum principle] \label{lemm:maximum_principle} 
		Let $\Omega \subset X$ be open. Then $\sup_{x \in \Omega \cap Q} u(x) \leq \sup_{y \in \partial_*\Omega} u(y)$ and $\inf_{x \in \Omega \cap Q} u(x) \geq \inf_{y \in \partial_*\Omega} u(y)$.
	\end{lemm}
	
	Lemma \ref{lemm:maximum_principle} allows us to establish topological properties for the closures of sets of the form $u^{-1}([s,t])$. 
	
	\begin{prop}
		\label{prop:connect}
		For all $s, t \in [0,1]$, $s \leq t$, the set $\overline{u^{-1}([s,t])}$ is connected and intersects both $\zeta_2$ and $\zeta_4$.
	\end{prop}
	\begin{proof}
		Let $E = \overline{u^{-1}([s,t])}$. To prove the first claim, suppose that $E$ is not connected. Then there is an open set $U \subset X$ such that 
		\begin{equation}
		\label{sussa}
		U \cap E \neq \emptyset, \quad (Q \setminus U) \cap E \neq \emptyset, \quad \partial U \cap E = \emptyset. 
		\end{equation}
		Let $E_1 = U \cap E$ and $E_2 = (Q \setminus U) \cap E$. By passing to a subset if needed, we may assume that $E_1$ and $E_2$ are each contained within a single component of $U$ and $Q \setminus \overline{U}$, respectively.  We fix $\varepsilon >0$ such that $\operatorname{dist}(\partial U,E) > \varepsilon$. By Proposition \ref{prop:coarea} applied to $h(x)=\operatorname{dist}(\partial U,x)$, there is $0<p<\varepsilon$ 
		such that $\mathcal{H}^1(h^{-1}(p))< \infty$ and every rectifiable curve $\gamma$ for which $|\gamma| \subset h^{-1}(p)$ lies outside the exceptional set $\Gamma_0$. By \eqref{sussa} and our 
		choice of $p$, the sets $E_1$ and $E_2$ belong to different components of $Q \setminus h^{-1}(p)$. Lemma \ref{lemm:separating_continuum} then shows that $h^{-1}(p)$ has a connected subset $F \subset Q$ such that $E_1$ and $E_2$ belong to different components of $Q \setminus F$. Notice that, for every rectifiable curve $\gamma$ with 
		$|\gamma| \subset F$, $u||\gamma|$ is continuous and either $u(x) < s$ or $u(x)>t$ for all $x \in |\gamma|$. We divide the rest of the proof into cases. 
		
		\begin{case}\label{pring}
			Suppose there is an open set $G \subset X$ such that $\partial G \subset F$ and $E_j \subset G$ for $j=1$ or $j=2$. By Lemma \ref{lemm:maximum_principle} there are $x_0,x_1 \in G$ such that $u(x_0)\leq s$ and $u(x_1)\geq t$. Moreover, by Proposition \ref{prop:existence_paths} there is a rectifiable curve $\gamma$ joining $x_0$ and $x_1$ in $F$. Since $u||\gamma|$ is continuous, we conclude that $u(x) \in E$ for some $x \in |\gamma|$. This is a contradiction, since $E \cap F = \emptyset$. 
		\end{case}
		
		Suppose next that the set $G$ in Case \ref{pring} does not exist. We then find a subcontinuum $F'$ of $F$ with the following properties: $F' \cap \partial Q$ consists of two distinct points $x_0$ and $x_1$, and $E_1$ and $E_2$ belong to different components, say $\Omega_1$ and $\Omega_2$, of $X \setminus (\partial Q \cup F')$. By Proposition \ref{prop:existence_paths} we may moreover assume that $F'=|\gamma|$, where $\gamma:[0,1]\to Q$ is simple and rectifiable, and $\gamma(0)=x_0$, $\gamma(1)=x_1$. 
		
		\begin{case}
			Suppose that both $x_0$ and $x_1$ belong to $\zeta_j$ for some $j=1,\ldots,4$. Then $\partial \Omega_k \subset |\gamma| \cup \zeta_j$ for $k=1$ or $k=2$. As in Case \ref{pring}, Lemma \ref{lemm:maximum_principle} and the continuity of $u||\gamma|$ show that there exists $x \in |\gamma|$ such that $u(x) \in [s,t]$. This contradicts the construction of $\gamma$. A similar argument can be applied when  
			$x_0 \in \zeta_i$ and $x_1 \in \zeta_j$, where either $i \in \{1,3\}$ and $j \in \{2,4\}$, or $j \in \{1,3\}$ and $i \in \{2,4\}$. 
		\end{case}
		
		\begin{case}
			Suppose that $x_0 \in \zeta_1$ and $x_1 \in \zeta_3$. Then, since $\gamma \notin \Gamma_0$, the construction of $u$ shows that $u||\gamma|$ takes all values between $0$ and $1$. In particular, 
			$u(x) \in [s,t]$ for some $x \in |\gamma|$. This contradicts the fact that $|\gamma| \cap E = \emptyset$. The argument remains valid if the roles of $x_0$ and $x_1$ are reversed. 
		\end{case}
		
		\begin{case}
			Suppose that $x_0 \in \zeta_2$ and $x_1 \in \zeta_4$. Without loss of generality, we may assume that $\Omega_1$ is the component containing $\zeta_1$. It then follows from Lemma 
			\ref{lemm:maximum_principle} that $u(x) \geq s$ for some $x \in |\gamma|$. Moreover, since $u||\gamma|$ is continuous and $|\gamma| \cap E = \emptyset$, it follows that in fact $u(x) >t$ for every $x \in |\gamma|$. Similarly, applying Lemma \ref{lemm:maximum_principle} to $\Omega_2$ shows that $u(x) < s$ for every $x \in |\gamma|$. This is a contradiction. The argument remains valid if the roles of $x_0$ and $x_1$ are reversed.   
		\end{case} 
		
		We conclude that the set $E$ is connected. It remains to show that $E$ intersects both $\zeta_2$ and $\zeta_4$. Suppose towards contradiction that this is not the case. We may assume without loss of generality that $E$ does not intersect $\zeta_4$. 
		Proposition \ref{prop:coarea} applied to $g(x)=\operatorname{dist}(\zeta_4,x)$ shows that there exists a small $p>0$ such that $\mathcal{H}^1(g^{-1}(p))< \infty$. Moreover, by Lemma \ref{lemm:separating_continuum} there is a continuum $F \subset g^{-1}(p)$ joining $\zeta_1$ and $\zeta_3$ in $Q \setminus E$. Proposition \ref{prop:existence_paths} gives a simple curve $\gamma$ such that $|\gamma| \subset F$ 
		also joins $\zeta_1$ and $\zeta_3$. As before, we may assume that $\gamma \notin \Gamma_0$ so that $u||\gamma|$ takes all values between $0$ and $1$. This is a contradiction since 
		$|\gamma| \cap E = \emptyset$. The proof is complete. 
	\end{proof}
	
	Next, we give a generalization of Lemma 15.6 in \cite{Raj:16}, with a corrected constant. The proof is essentially the same as the corresponding proof in \cite{Raj:16}.
	
	\begin{lemm}\label{lemm:oscillation} 
		Let $x \in Q$ and $r \in (0, r_0)$, where $r_0 = \min\{\diam \zeta_1, \diam \zeta_3\}/4$. Then 
		\begin{equation} \label{equ:oscillation_bound}
		r  \mathcal{H}^1(u(B(x,r) \cap Q))\leq \frac{4}{\pi} \int_{B(x,2r)} \rho\, d\mathcal{H}^2.
		\end{equation} 
		Moreover, if $U(x,r)$ is the $x$-component of  $B(x,r) \cap Q$, then  
		\begin{equation} \label{equ:oscillation_bound2}
		r  \osc_{U(x,r)} u \leq \frac{4}{\pi} \int_{B(x,2r)} \rho\, d\mathcal{H}^2.
		\end{equation} 
	\end{lemm}
	\begin{proof}	
		By applying Proposition \ref{prop:coarea} to the function $d(\cdot,x)$ and arguing as in the first paragraph of the proof of Proposition \ref{prop:connect}, we see that for almost every $s \in (r,2r)$, the sphere $S(x,s)$ satisfies $\mathcal{H}^1(S(x,s)) < \infty$ and has the property that $\eta \notin \Gamma_0$ for every curve $\eta$ with $|\eta| \subset S(x,s) \cap Q$. Fix such an $s \in (r,2r)$.
			
			Then $B(x,s) \cap Q$ consists of countably many relatively open components $V_j$. By Lemma \ref{lemm:separating_continuum}, for such a component $V_j$ there is a simple curve $\gamma_j$ with $|\gamma_j| \subset S(x,s)$ that separates $Q$ into the relative components $U_j$ and $Q \setminus \overline{U}_j$, where $V_j \subset U_j$. Observe that either $\gamma_j$ is a closed curve, or the two endpoints of $\gamma_j$ are contained in $\partial Q$. 
			
			Since $B(x,r) \cap Q \subset \bigcup_j U_j$, we have
			$$\mathcal{H}^1(u(B(x,r) \cap Q)) \leq \sum_j \diam u(U_j). $$
			By the maximum principle Lemma \ref{lemm:maximum_principle},
			$$\diam u(U_j) \leq \sup_{y,z \in \partial_* U_j} |u(y) - u(z)|.$$
			By our assumption that $r \leq \min\{ \diam \zeta_1, \diam \zeta_3\}/4$, it follows that if $\zeta_1 \cap \partial_* U_j \neq \emptyset$, then there exists a point $z_1 \in |\gamma_j| \cap \zeta_1$. Indeed, if $y \in \zeta_1 \cap \partial_* U_j$, then $d(y,x) \leq 2r$. But by assumption, there exists $z \in \zeta_1$ such that $d(y,z) > 4r$. The triangle inequality gives $d(z,x) > 2r$, and in particular $z \notin \overline{U}_j$. Since $\gamma_j$ separates $Q$, we conclude there is a point $z_1 \in |\gamma_j| \cap \zeta_1$. In this case it follows that $0 = \inf_{z \in \partial_* U_j} u(z) = u(z_1) = \min_{z \in |\gamma_j|} u(z)$. On the other hand, if $\zeta_1 \cap \partial_* U_j = \emptyset$, then by Lemma \ref{lemm:maximum_principle} we again have $\inf_{z \in \partial_* U_j} u(z) = \min_{z \in |\gamma_j|} u(z)$.
			
			The same argument shows that if $\zeta_3 \cap \partial_* U_j \neq \emptyset$, then there exists $y_1 \in |\gamma_j| \cap \zeta_3$ such that $1 = \sup_{y \in \partial_* U_j} u(y) = u(y_1) = \max_{y \in \gamma_j} u(y)$. In general, we likewise have $\sup_{y \in \partial_* U_j} u(y) = \max_{y \in |\gamma_j|} u(y)$. This establishes the equality 
			$$\sup_{y,z \in \partial_* U_j} |u(y) - u(z)| = \max_{y,z \in |\gamma_j|} |u(y) - u(z)| .$$ 
			By the upper gradient inequality \eqref{equ:upper_gradient},
			$$\max_{y,z \in |\gamma_j|} |u(y) - u(z)| \leq \int_{\gamma_j} \rho\,d\mathcal{H}^1.$$
			Finally, combining the estimates gives
			$$\mathcal{H}^1(u(B(x,r) \cap Q)) \leq \sum_j \diam u(U_j) \leq \sum_j \int_{\gamma_j} \rho\, d\mathcal{H}^1 \leq \int_{S(x,s)} \rho\, d\mathcal{H}^1. $$
			Observe that this estimate is the same independent of our choice of $s$. Inequality \eqref{equ:oscillation_bound} then follows from integrating over $s$ from $r$ to $2r$ and applying Proposition \ref{prop:coarea}.
			
			The same argument also verifies inequality \eqref{equ:oscillation_bound2}, since for each choice of $s \in (r,2r)$ it holds that $\osc_{U(x,r)} u = \diam u(U(x,r)) \leq \sum_j \diam u(U_j)$. 
	\end{proof}
	
	Without assuming the reciprocality conditions, it is not clear that the function $u$ is continuous. Nevertheless, Lemma \ref{lemm:oscillation} implies a certain amount of continuity for $u$, as we show in the following corollary.
	
	\begin{cor} \label{cor:continuity}
		The function $u$ is continuous at $\mathcal{H}^2$-almost every $x \in Q$.
	\end{cor}
	\begin{proof}
		Inequality \eqref{equ:oscillation_bound2} implies that
		$$\limsup_{r \rightarrow 0}  \osc_{U(x,r)} u \leq \limsup_{r \rightarrow 0} \frac{4r}{\pi}\cdot \frac{1}{r^2} \int_{B(x,2r)} \rho\, d\mathcal{H}^2$$
		for all $x \in Q \setminus \partial Q$. Here, $U(x,r)$ is as in Lemma \ref{lemm:oscillation}.
		From basic properties of pointwise densities of measures (see \cite[Sec. 2.10.19(3)]{Fed:69}), the integrability of $\rho$ and local finiteness of $\mathcal{H}^2$ imply that $$ \limsup_{r \rightarrow 0} \frac{1}{r^2} \int_{B(x,2r)} \rho\, d\mathcal{H}^2 < \infty$$
		for $\mathcal{H}^2$-almost every $x \in Q$. The result follows by combining the estimates. 
	\end{proof}
	
	
	\section{Reciprocal lower bound} \label{sec:lower_bound} 
	
	This section is devoted to a proof of Theorem \ref{thm:main}. We first state and prove the coarea inequality mentioned above which constitutes the main technical contribution of this paper. This corresponds to Proposition 15.7 in \cite{Raj:16}, where a similar result is proved under the assumption that $X$ has the mass upper bound $\mathcal{H}^2(B(x,r)) \leq Cr^2$. The proof of Proposition \ref{prop:coarea_u}, like Proposition 15.7 in \cite{Raj:16}, is based on standard arguments such as that in \cite[Prop. 3.1.5]{AmbTil:04}. 
	
	\begin{prop}\label{prop:coarea_u}
		Let $u$ and $\rho$ be as above. For all Borel functions $g: Q \rightarrow [0,\infty]$, 
		$$
		\int_{[0,1]}^* \int_{\overline{u^{-1}(t)}} g\,d\mathcal{H}^1\,dt \leq 2000 \int_Q g\rho\, d\mathcal{H}^2. 
		$$
		Here $\int^*_A a(t)\, dt $ is the upper Lebesgue integral of $a$ over $A$ (see \cite[Sec. 2.4.2]{Fed:69}). 
	\end{prop}
	\begin{proof}
		It suffices to consider the case where $g$ is a characteristic function, that is, $g = \chi_E$ for some Borel set $E \subset Q$. Moreover, we may assume that $E$ is open in $Q$. Indeed, for a Borel set $E$ we find open sets $U_j \supset E$, $U_{j+1} \subset U_j$, such that $\mathcal{H}^2(U_j) \to \mathcal{H}^2(E)$. Assuming the proposition for $g=\chi_{U_j}$, we have 
		\begin{eqnarray*}
			\int_{[0,1]}^* \int_{\overline{u^{-1}(t)}} \chi_E \,d\mathcal{H}^1\,dt  &\leq &\int_{[0,1]}^* \int_{\overline{u^{-1}(t)}} \chi_{U_j} \,d\mathcal{H}^1\,dt \leq 2000 \int_Q \chi_{U_j}\rho\, d\mathcal{H}^2 \\ 
			&\longrightarrow& 2000 \int_Q \chi_{E}\rho\, d\mathcal{H}^2. 
		\end{eqnarray*}
		So we want to show that 
		\begin{equation} \label{equ:set_E} 
		\int_{[0,1]}^* \mathcal{H}^1(\overline{u^{-1}(t)} \cap E) \,dt \leq 2000 \int_E \rho\,d\mathcal{H}^2 
		\end{equation}
		whenever $E$ is open in $Q$. The proof is divided into two steps, the first dealing with the subset of ``good'' points of $E$ and the second dealing with the subset of ``bad'' points. Throughout this proof, all metric balls are considered as subsets of $Q$. 
		
		\begin{step}  
			Consider the set $$G = \left\{x \in E: \forall \varepsilon>0, \exists r<\varepsilon, \int_{B(x,10r)} \rho\,d\mathcal{H}^2 \leq 200 \int_{B(x,r)} \rho \,d\mathcal{H}^2\right\}.$$ 
			
			Fix $\varepsilon>0$. We apply the basic covering theorem (\cite[Thm. 1.2]{Hei:01}) to choose a 
			countable collection of pairwise disjoint balls $B_j = B(x_j,r_j)$ such that $x_j \in G$ and $10r_j \leq \min\{\varepsilon, d(x_j, Q\setminus E)\}$ for each $j$, the collection $\{5B_j\}$ covers $G$, and 
			$$\int_{10B_j} \rho\,d\mathcal{H}^2 \leq 200 \int_{B_j} \rho \,d\mathcal{H}^2 $$
			for each $j$. We also require that $20r_j < \min\{\diam \zeta_1, \diam \zeta_3\}$ for our application of Lemma \ref{lemm:oscillation}. We have
			$$\sum_j \int_{10B_j} \rho\, d\mathcal{H}^2 \leq \sum_j 200\int_{B_j} \rho\, d\mathcal{H}^2 \leq 200\int_{E} \rho\,d\mathcal{H}^2,$$
			where the last inequality follows since by our choice the balls $B_j$ are pairwise disjoint subsets of the open set $E$. For each $j$ fix a measurable set $A_j \supset u(5B_j)$ such that $\mathcal{H}^1(A_j)=\mathcal{H}^1(u(5B_j))$. Moreover, 
			define $g_\varepsilon: [0,1] \rightarrow \mathbb{R}$ by $$g_\varepsilon(t) = \sum_j r_j \chi_{A_j}(t).$$ 
			Integrating and applying Lemma \ref{lemm:oscillation} gives 
			$$\int_0^1 g_\varepsilon(t)\,dt = \sum_j r_j \mathcal{H}^1(u(5B_j)) \leq \frac{4}{\pi}\sum_j \int_{10B_j} \rho\, d\mathcal{H}^2.$$
			We observe that if $x \in \overline{u^{-1}(t)} \cap G$ for a given $t \in [0,1]$, with $j_x$ such that $x \in 5B_{j_x}$, then of necessity $t \in u(5B_{j_x})$. Hence $\mathcal{H}_\varepsilon^1(\overline{u^{-1}(t)} \cap G) \leq 10g_\varepsilon(t)$, by the definition of Hausdorff $\varepsilon$-content. Letting $\varepsilon \rightarrow 0$ and applying Fatou's lemma gives 
			$$
			\int_{[0,1]}^* \mathcal{H}^1(\overline{u^{-1}(t)} \cap G)\, dt \leq 10 \int_0^1 \liminf_{\varepsilon \to 0} g_\varepsilon(t)\, dt \\
			\leq 10 \liminf_{\varepsilon \rightarrow 0} \int_0^1 g_\varepsilon(t) \, dt. 
			$$
			Combining estimates, we obtain
			$$ \int_{[0,1]}^*\mathcal{H}^1(\overline{u^{-1}(t)} \cap G)\, dt \leq \frac{4 \cdot 2000}{\pi} \int_E \rho\,d\mathcal{H}^2.$$
		\end{step}
		\begin{step}
			We turn our attention next to the set $F = E \setminus G$. We claim that \begin{equation} \label{equ:bad_points}
			\int_{[0,1]}^* \mathcal{H}^1(\overline{u^{-1}(t)} \cap F)\,dt = 0.
			\end{equation} 
			By the definition of $F$, for all $x \in F$ there exists $\varepsilon_x = 10^{-k_x}$ (for some integer $k_x \geq 1$) such that  
			\begin{equation} \label{equ:bad_iteration}
			\int_{B(x,10^{-j})} \rho\,d\mathcal{H}^2  \leq 200^{-1} \int_{B(x,10^{-j+1})} \rho\,d\mathcal{H}^2 \leq \cdots \leq 200^{-(j-k_x)} \int_{B(x,\varepsilon_x)} \rho\,d\mathcal{H}^2
			\end{equation}
			for all $j \geq k_x$. For all $k \in \mathbb{N}$, let $F_k = \{x \in F: k_x \leq k\}$. Observe that $F = \bigcup_k F_k$.   
			
			Now, fix $k \in \mathbb{N}$ and let $j \geq k$. By definition of the (spherical) Hausdorff measure, there exists a countable collection of balls $B_m=B(x_m, r_m)$ which cover $F_k$, such that $x_m \in F_k$, $r_m \leq \min\{10^{-j}, d(x_m,Q \setminus E),\diam\zeta_1/4, \diam \zeta_3/4\}$, and $\sum 4r_m^2 \leq 4\mathcal{H}^2(F_k)+4/j$.  For the last requirement, recall that the spherical Hausdorff 2-measure is at most 4 times the usual Hausdorff 2-measure. For each $m$, let $j_m$ be the largest integer such that $2r_m \leq 10^{-j_m}$. Observe that $10^{-j_m} < 20r_m \leq 20 \cdot 10^{-j}$ and hence that $j_m \geq j-1$.
			
			From Lemma \ref{lemm:oscillation} and \eqref{equ:bad_iteration} we deduce
			\begin{align*}  
			r_m \mathcal{H}^1(u(B_m))& \leq \frac{4}{\pi} \int_{2B_m} \rho\,d\mathcal{H}^2 \leq \frac{4}{\pi} \int_{B(x,10^{-j_m})}\rho\,d\mathcal{H}^2  \\ 
			& \leq \frac{4}{\pi} \cdot \frac{1}{200} \int_{B(x,10^{-j_m+1})}\rho\,d\mathcal{H}^2 \\ 
			& \leq \cdots \leq \frac{4}{\pi} \cdot \frac{1}{200^{j_m-k}} \int_{B(x,10^{-k})} \rho\, d\mathcal{H}^2.
			\end{align*}
			In particular,
			\begin{equation} \label{eq:oscillation}
			r_m \mathcal{H}^1(u(B_m)) \leq \frac{4}{\pi} \cdot \frac{200^k}{200^{j_m}} \int_Q \rho\,d\mathcal{H}^2.
			\end{equation}
			
			Similar to the first step of the proof, for each $m$ fix a measurable $A_m \supset u(B_m)$ such that $\mathcal{H}^1(A_m)=\mathcal{H}^1(u(B_m))$ and define $g_j(t) =  \sum_m r_m\chi_{A_m}(t)$. Then, as before, the definition of $\mathcal{H}_{1/j}^1$ gives 
			\begin{equation}
			\label{nakki} 
			\mathcal{H}_{1/j}^1(\overline{u^{-1}(t)} \cap F_k) \leq 2g_j(t) 
			\end{equation}
			for all $t \in [0,1]$. Integrating gives
			$$\int_0^1 g_j(t)\,dt \leq \sum_m r_m\mathcal{H}^1(u(B_m)). $$
			Applying \eqref{eq:oscillation} and using the relationships $1 < 20\cdot 10^{j_m} r_m$ and $j_m \geq j-1$ gives
			\begin{align*}
			\sum_m r_m\mathcal{H}^1(u(B_m)) & \leq \sum_m \frac{3200}{\pi}\cdot 200^k r_m^2 \left( \frac{100}{200} \right)^{j_m} \int_Q \rho\,d\mathcal{H}^2 \\
			& \leq \frac{3200}{\pi}\cdot 200^k \left( \frac{100}{200} \right)^{j} \left(\int_Q \rho\,d\mathcal{H}^2 \right) \sum_m r_m^2 \\
			& \leq \frac{3200}{\pi}\cdot 200^k \left( \frac{100}{200} \right)^{j} \left(\int_Q \rho\,d\mathcal{H}^2 \right)\left(\mathcal{H}^2(F_k)+1/j \right) .
			\end{align*}
			From this we obtain 
			$$\lim_{j \to \infty} \int_0^1 g_j(t) \, dt \leq \lim_{j \to \infty} \frac{3200}{\pi}\cdot 200^k \cdot2^{-j}  \left(\int_Q \rho\,d\mathcal{H}^2 \right)\left(\mathcal{H}^2(F_k)+1/j \right)=0. $$
			Combining with Fatou's lemma and \eqref{nakki} shows that $ \mathcal{H}^1(\overline{u^{-1}(t)} \cap F_k)=0$ for almost every $t$. Since this is true for all $k$, 
			\eqref{equ:bad_points} follows.
		\end{step} 
	\end{proof}
	
	With Proposition \ref{prop:coarea_u} in hand, the proof of Theorem \ref{thm:main} is now simple.  
	\begin{proof}[Proof of Theorem \ref{thm:main}]
		First, observe from Proposition \ref{prop:coarea_u} that $\mathcal{H}^1(\overline{u^{-1}(t)}) < \infty$ for almost every $t \in [0,1]$. Also, as shown in Proposition \ref{prop:connect}, $\overline{u^{-1}(t)}$ is connected for all $t$ and connects $\zeta_2$ and $\zeta_4$. By Proposition \ref{prop:existence_paths}, for almost every $t \in [0,1]$, $\overline{u^{-1}(t)}$ contains a simple rectifiable curve $\gamma_t$ joining $\zeta_2$ and $\zeta_4$ in $Q$. Let $g: Q \rightarrow [0,\infty]$ be an admissible function for $\Gamma_2$. Then 
		\begin{equation}
		\label{nakka}
		1 \leq \int_{\gamma_t} g \, ds \leq \int_{\overline{u^{-1}(t)}} g\,d\mathcal{H}^1
		\end{equation}
		for almost every $0 \leq t \leq 1$. Combining \eqref{nakka} with Proposition \ref{prop:coarea_u} yields 
		$$1 \leq \int^*_{[0,1]}\int_{\overline{u^{-1}(t)}} g\,d\mathcal{H}^1\,dt \leq  \frac{4 \cdot 2000}{\pi} \int_Q g\rho\, d\mathcal{H}^2.$$
		By H\"older's inequality,
		$$\int_Q g\rho\,d\mathcal{H}^2 \leq \left( \int_Q g^2\,d\mathcal{H}^2 \right)^{1/2} \left( \int_Q \rho^2\, d\mathcal{H}^2 \right)^{1/2} = \left( \int_Q g^2\,d\mathcal{H}^2 \right)^{1/2} (\Mod \Gamma_1)^{1/2}.$$
		Infimizing over all admissible $g$, we obtain
		$$\frac{1}{2000^2\cdot (4/\pi)^2} \leq \Mod \Gamma_1 \cdot \Mod \Gamma_2.$$
	\end{proof}
	
	\begin{rem}
		We can improve the value of $\kappa$ as follows. For $\delta>0$, a version of the basic covering theorem yields a family of balls $B_j$ with the property that $\{(3+\delta)B_j\}$ covers $G$, instead of $\{5B_j\}$. In the definition of the set $G$ in Proposition \ref{prop:coarea_u}, we may then use $B(x,2(3+\delta)r)$ in place of $B(x,10r)$. We also replace the constant 200 with $4(3+\delta)^2 + \delta$. Following the remainder of the proof and letting $\delta \rightarrow 0$ yields the final value of $\kappa = 216^2\cdot (4/\pi)^2$. 
	\end{rem} 
	
	\section{Continuity of $u$} \label{sec:continuity_u}
	
	In this section, we strengthen Corollary \ref{cor:continuity} by showing that the harmonic function $u$ is continuous on the entire set $Q$. In Theorem 5.1 of \cite{Raj:16}, the continuity of $u$ is proved employing reciprocality condition \eqref{equ:reciprocality(3)}. In contrast, we do not assume any of the reciprocality conditions in this section. 
	
	First, we need a technical fact. This is proved using Proposition 3.1 in \cite{Raj:16} (which is a re-statement of Proposition 15.1 in \cite{Sem:96c}) and an induction and limiting argument. 
	
	\begin{prop} \label{prop:curve_parametrization}
		Let $X$ be a metric space and $E \subset X$ a continuum with $\mathcal{H}^1(E) < \infty$. For all $x, y \in E$, there is a 1-Lipschitz curve $\gamma: [0, 2\mathcal{H}^1(E)] \rightarrow E$ such that $|\gamma| = E$, $\gamma(0) = x$, $\gamma(2\mathcal{H}^1(E)) = y$, and $\gamma^{-1}(z)$ contains at most two points for $\mathcal{H}^1$-almost every $z \in E$. 
	\end{prop}
	\begin{proof}
		For this proof, we will let $D$ denote the length metric on $E$ induced by $d$. We write $D_{zw}$ in place of $D(z,w)$. Observe that $D_{zw} < \infty$ for all $z,w \in E$ by Proposition 3.1 in \cite{Raj:16}. Also, for $z,w \in E$, we use $\gamma_{zw}$ to denote some fixed choice of injective 1-Lipschitz curve in $E$ from $z$ to $w$ whose length attains $D_{zw}$; the existence of at least one such curve is guaranteed by the Hopf-Rinow theorem. Let $L = 2\mathcal{H}^1(E)$.  
		
		We will inductively define a sequence of curves $\gamma_j: [0, L] \rightarrow E$. We define first $\gamma_1$ by
		$$\gamma_1(t) = \left\{ \begin{array}{ll} \gamma_{xy}(t) & 0 \leq t \leq D_{xy} \\ y & D_{xy} \leq t \leq L \end{array} \right. .$$
		
		For the inductive step, assume that $\gamma_j$ has been defined for some $j \in \mathbb{N}$. If $|\gamma_j| = E$, then stop and take $\gamma = \gamma_j$. Otherwise, define $\gamma_{j+1}$ as follows. Let $z_j$ be a point in $E$ maximizing $D$-distance from $|\gamma_j|$. Such a point exists by the compactness of $E$. Let $\gamma_{w_jz_j}$ be a shortest curve from $|\gamma_j|$ to $z_j$, with initial point $w_j \in |\gamma_j|$. Let $t_j$ denote the smallest point in $[0, L]$ for which $\gamma_j(t_j) = w_j$. Define now $\gamma_{j+1}$ by
		$$\gamma_{j+1}(t) = \left\{ \begin{array}{ll} \gamma_j(t) & 0 \leq t \leq t_j \\ \gamma_{w_jz_j}(t-t_j) & t_j \leq t \leq t_j + D_{w_jz_j} \\ \gamma_{w_jz_j}(t_j+2D_{w_jz_j} - t) & t_j + D_{w_jz_j} \leq t \leq t_j + 2D_{w_jz_j} \\ \gamma_j(t-2D_{w_jz_j}) & t_j + 2D_{w_jz_j} \leq t \leq \ell(\gamma_j) + 2D_{w_jz_j} \\ y & \ell(\gamma_j) + 2D_{w_jz_j} \leq t \leq L \end{array} \right. .$$

		Observe that the curve $\gamma_j$ has multiplicity at most 2, except possibly at the points $w_j$. Thus $\ell(\gamma_j) + 2D_{w_jz_j} \leq D_{xy} + \sum_{k=1}^{j} 2 D_{w_kz_k} < 2\mathcal{H}^1(|\gamma_j|) \leq L$. Hence the curve $\gamma_{j+1}$ is well-defined.  
		
		We also note that $D(\gamma_{j+1}(t),\gamma_j(t)) \leq 2D_{w_jz_j}$ for all $t \in [0,L]$ and $j \in \mathbb{N}$, and thus the curves $\gamma_j$ converge pointwise to a curve $\gamma: [0, L] \rightarrow E$. By construction, the curve $\gamma$ has multiplicity at most 2, except possibly on the countable set $\{w_j\}$. To see that $|\gamma| = E$, suppose there exists $z \in E \setminus |\gamma|$. But then $D(z,|\gamma|) > 0$. In particular, there exists $j \in \mathbb{N}$ with $D(w_j,z_j) < D(z, |\gamma_j|)$, contradicting the maximality of the choice of $z_j$.  
	\end{proof}
	
	We proceed now to the main result of this section. 
	
	\begin{thm} \label{thm:continuity}
		The function $u$ is continuous in $Q$. 
	\end{thm}
	\begin{proof}
		For all $t \in [0,1]$ such that $\mathcal{H}^1(\overline{u^{-1}(t)}) < \infty$, let $\gamma_t$ denote a curve connecting $\zeta_2$ to $\zeta_4$ whose image is $\overline{u^{-1}(t)}$ satisfying the conclusions of Proposition \ref{prop:curve_parametrization}.  	 
		By Lemma 4.3 in \cite{Raj:16}, $u$ is continuous on each $\gamma_t$ except on a curve family of modulus zero.  Observe that $$\int_{\gamma_t} g\,ds \leq 2\int_{\overline{u^{-1}(t)}}g\, d\mathcal{H}^1$$ 
		for each $t$ such that $\gamma_t$ is defined, for any Borel function $g:Q \rightarrow [0, \infty]$. From this fact and the coarea inequality Proposition \ref{prop:coarea_u}, it follows that $u$ is continuous on $\gamma_t$ for every $t \in E$, where $E \subset [0,1]$ has full measure. 
		
		Suppose for contradiction that $u$ is not continuous at the point $x \in Q$. Let $s_1 = \liminf_{y \rightarrow x} u(y)$ and $s_2 = \limsup_{y \rightarrow x} u(y)$; then $0 \leq s_1 < s_2 \leq 1$. Take $\varepsilon$ satisfying $0 < \varepsilon < (s_2-s_1)/2$. Then $x \in A_1 \cap A_2$, where $A_1= \overline{u^{-1}([s_1-\varepsilon, s_1+\varepsilon])}$ and $A_2 = \overline{u^{-1}([s_2-\varepsilon,s_2+\varepsilon])}$. Pick $t_1,t_2 \in (s_1 + \varepsilon, s_2 - \varepsilon) \cap E$ with $t_1 < t_2$. 
		Observe that $Q \setminus |\gamma_{t_1}|$ consists of two disjoint relatively open sets $U_1, U_2 \subset Q$, where each component of $U_1$ intersects $\zeta_1$ and each component of $U_2$ intersects $\zeta_3$. Lemma \ref{lemm:maximum_principle} implies that $A_1 \subset \overline{U}_1$ and that $A_2 \subset \overline{U}_2$. This shows that $x \in \overline{U}_1 \cap \overline{U}_2$ and hence that $x \in |\gamma_{t_1}|$. Since $u^{-1}(t_1)$ is a dense subset of $|\gamma_{t_1}|$, we see that $u(x) = t_1$. However, the same argument shows that $u(x) = t_2$, giving a contradiction. 
	\end{proof}
	
	\noindent
	{\bf Acknowledgement.}
	We are grateful to Toni Ikonen, Atte Lohvansuu, Dimitrios Ntalampekos, Martti Rasimus and the referee for their comments and corrections.

	\bibliographystyle{abbrv}
	\bibliography{ReciprocalLowerBoundBiblio} 
\end{document}